\documentclass[oneside,english,reqno, 11pt]{amsart}
\usepackage[T1]{fontenc}
\usepackage[latin9]{inputenc}
\usepackage{amsthm}
\usepackage{amssymb}
\usepackage{amsmath}
\usepackage{mathabx}
\usepackage{xcolor}
\usepackage{verbatim, color}
\usepackage{bbm}
\usepackage{float}
\usepackage{dsfont, graphicx}
\usepackage{epstopdf}
\usepackage[T1]{fontenc}
\usepackage{amssymb,amsmath, amsthm, amsfonts}
\usepackage{graphicx}

\usepackage[hidelinks]{hyperref}

\usepackage{thmtools,thm-restate}

\newtheorem{theorem}{Theorem}[section]
\newtheorem{thm}{Theorem}[section]
\newtheorem{lemma}[theorem]{Lemma}
\newtheorem{proposition}[theorem]{Proposition}

\newcommand{\R}{\mathbb{R}}

\newcommand{\f}{\frac}
\newcommand{\beq}{\begin{equation}}
\newcommand{\eeq}{\end{equation}}
\newcommand{\beqq}{\begin{equation*}}
\newcommand{\eeqq}{\end{equation*}}

\theoremstyle{definition}
\newtheorem{definition}[theorem]{Definition}

\theoremstyle{remark}
\newtheorem{remark}[theorem]{Remark}

\numberwithin{equation}{section}

\begin{document}

\address{Chenjie Fan
\newline \indent Academy of Mathematics and Systems Science, CAS\indent 
\newline \indent  
Beijing, China.\indent }
\email{cjfanpku@gmail.com}

\address{Zehua Zhao
\newline \indent Department of Mathematics, University of Maryland\indent 
\newline \indent  William E. Kirwan Hall, 4176 Campus Dr. College Park, MD 20742-4015\indent }
\email{zzh@umd.edu}

\title[stochastic nonlinear Schr\"odinger equations with a multiplicative noise]{On long time behavior for stochastic nonlinear Schr\"odinger equations with a multiplicative noise}
\author{Chenjie Fan and Zehua Zhao}
\maketitle

\begin{abstract}
In this article, we study  Stochastic mass critical nonlinear Schr\"odinger equations with a multiplicative noise in 3D with a slight time decay ($\langle t \rangle^{-\epsilon}$), and prove associated space-time bound and scattering behavior.
\end{abstract}

\section{Introduction}
\subsection{Statement of main results}
In this article, we aim to study the long time behavior for Stochastic mass critical nonlinear Schr\"odinger equations with a multiplicative noise,
\begin{equation}\label{eq: snls}
\begin{cases}
du=-i(\Delta u-|u|^{\frac{4}{d}}u)dt-iu \langle t \rangle ^{-\gamma}dW_{t}-\frac{1}{2}\langle t \rangle^{-2\gamma}W^{2}udt,
u_{0}\in L_{\omega}^{\infty}L^2_{x}(\mathbb{R}^{d})
\end{cases}
\end{equation}

We will focus\footnote{We follow the convention for Brownian motion, we use $B_{t}$ rather than $B(t)$ to denote the value of $B$ at $t$.} on case $W(x,t)=V(x)B_{t}$, $V$ Schwarz and real.  We have also impose an extra time decay $\langle t \rangle^{-\gamma}$.

Though the noise is of such a simple form, very few (except for global existence) seem to be known about the long time dynamic of \eqref{eq: snls}, if $\gamma=0$, i.e. when no extra time decay of noise is imposed.

Indeed, if $\gamma=0$, even the linear case seems not well understood. See \cite{FX} for a linear estimate while extra smallness (but no time decay ) is imposed to the noise. 

On the other hand, if one take $\gamma>1$, then it should be not surprising that one may reduce the study of \eqref{eq: snls} to a local theory, since 
\begin{equation}
\int_{0}^{\infty}\langle s \rangle^{-\gamma}ds<\infty, \gamma>1.
\end{equation}
The next natural  threshold will be $\gamma>1/2$. Since from the view point of Burkholder, one has 
\begin{equation}
\int_{0}^{\infty} \langle s \rangle^{-\gamma}dB_{s}\sim (\int_{0}^{\infty}\langle t \rangle^{-2\gamma}dt)^{1/2}<\infty, \gamma>1/2.
\end{equation}
This view point was illustrated (from a different perspective) in the very interesting recent work of \cite{HRZ}, which covers noise of finite quadratic variation in time. In particular, If one replaces $\langle t \rangle^{-\gamma}$ in \eqref{eq: snls} by $g(t)$ such that 
\begin{equation}
\int |g(t)|^{2}dt<\infty,
\end{equation}
then scattering behavior can be obtained\footnote{Both our work and \cite{HRZ} does not need $V$ be Schwarz and the noise is not necessarily of product form, $V(x)B_{t}$. We neglect all those technical issues for the moment and focus on the problem of reducing the time decay need.}.
 This, in particular, includes all $\gamma>1/2$. We note that the work of \cite{HRZ} also covers energy based model.

One very interesting and highly nontrivial problem is to understand what will happen if one poses $\gamma=0$, but this seems very hard even for $W(x,t)=V(x)B(t)$, $V$ Schwarz and small, with Schwarz and small initial data.

In the current article, we prove that scattering behavior exists for all $\gamma>0$ in dimension 3, for all $L_{x}^{2}$ data. Our main result is 
\begin{thm}\label{thm: main}
When the dimension $d=3$, for $\gamma>0$, let $u$ be the global solution to \eqref{eq: snls}, we have for all $(\alpha_{0},\beta_{0})$ which are admissible Strichartz pair and $\alpha_{0}>4$, 
\begin{equation}\label{eq: spacetimebound3d}
 \|u\|_{L_{\omega}^{\alpha_{0}}L_{t}^{\alpha_{0}}L_{x}^{\beta_0}([0,\infty)\times \mathbb{R}^{3})}\lesssim_{\|u_{0}\|_{L_{\omega}^{\infty}L_{x}^{2}}, \alpha_{0},\gamma} 1.
\end{equation}
and further more, the solution scatters in the sense that, there exists $u^{+}\in L_{\omega}^{\infty}L_{x}^{2}$, such that 
\begin{equation}\label{eq: scattering3d}
\|u(t)-e^{it\Delta}u^{+}\|_{L_{\omega}^{\rho^{*}}L_{x}^{2}} \rightarrow 0, \text{ as }t \rightarrow \infty, \text{ for some } \rho^{*}>1.
\end{equation}
In particular, one has scattering in probability.
\begin{equation}\label{eq: conpro}
\forall \delta>0, \lim_{t\rightarrow \infty}\mathbb{P}\left(\|u(t)-e^{it\Delta}\|_{L_{x}^{2}}\geq \delta\right)\rightarrow 0.
\end{equation}
and we indeed have scattering almost surely,
\begin{equation}\label{eq: almostsure}
\|u(t)-e^{it\Delta}u^{+}\xrightarrow{t\rightarrow \infty}\|_{L^2_x} 0, a.s.
\end{equation}
\end{thm}

The global existence of such solutions in Theorem \ref{thm: main} already follows from \cite{fan2018global}, \cite{DZhang}.\footnote{\cite{fan2018global} was written for dimension 1, but one can easily generalize it to high dimensions.}  In particular, we know for all\footnote{Though \eqref{eq: start} is a local bound, but unlike usual local theory, no matter how small is $T$, very nonlinear dynamic may still happen (with small probability though), and the theory is not purely perturbative.}  $T$, one has for all $\rho<\infty$
\begin{equation}\label{eq: start}
\|u\|_{L_{\omega}^{\rho}L_{t}^{\alpha_{0}}L_{x}^{\beta_{0}}[0,T]}\lesssim_{\rho, T, \|u(0)\|_{L_{\omega}^{\infty}L_{x}^{2}}} 1.
\end{equation}
And mass conservation law gives an extra $L_{\omega}^{\infty}L_{t}^{\infty}L_{x}^{2}$ bound. It allows one to get a global flow. It allows one to get a continuous range of  Strichartz type bound.
See also \cite{dBD03} on  the global existence for $H^{1}$ initial data.  
To the best of our knowledge, even for $\gamma=\frac{1}{2}$ and $u_{0}$ some Schwarz function, our result regarding \eqref{eq: snls} is new.

 Unlike the deterministic case, we do not see that global Strichartz spacetime bound \eqref{eq: spacetimebound3d} directly implies the scattering asymptotic. And extra efforts are needed to upgrade \eqref{eq: spacetimebound3d} into \eqref{eq: scattering3d}.

One can use the language of $\gamma-$randonifying operator to generalize the Theorem \ref{thm: main} to more general noise with essentially same proof\footnote{Roughly speaking, our main result Theorem \ref{thm: main} hold for $V$ such that $V$ are both decay in space and frequency in some sense. In particular all Schwarz functions satisfy this property. Let us say this can be characterized by some norm $\|\cdot \|_{A}$, then automatically the result will hold for noise $\sum_{k}V_{k}B_{k}$ where $\sum_{k}\|V_{k}\|_{A}<\infty$ and $B_{k}$ are i.i.d Brownian Motions, and same extra time decay $\langle t \rangle^{-\alpha}$ is imposed. The language of $\gamma-$randonifying operator allows one to relax the $l^{1}$ summability to $l^{2}$ summability by exploring the property of Gaussian, this is by now standard, one may refer to \cite{dBD99} for more details.}.

 The analogue of Theorem \ref{thm: main} holds for the focusing case if one further impose $\|u_0\|_{L_{\omega}^{\infty}L^2_x}<\|Q\|_{L^2_x}$, where $Q$ is the unique radial solution to 
\begin{equation}
-\Delta Q+Q=|Q|^{\frac{4}{3}}Q.
\end{equation}
In the literature, $Q$ is usually called ground state. The proof is essentially same except one needs to recall the analogues of Theorem \ref{thm: dod} is established in \cite{dodson2015global}, see also \cite{weinstein1983nonlinear}.

We also expect same technique gives same result in for mass critical models in dimension $d\geq 2$. However, if one want to further explore the problem for $\gamma=0,$ the most favorable dimension\footnote{Conceptually, dimension 1 and 2 are harder, but extra technical problems also rise in high dimensions. } seems to be $3$. 

\subsection{Background}
We start with the deterministic (defocusing) mass-critical nonlinear Schr\"odinger equation.
\begin{equation}\label{eq: mnls}
\begin{cases}
iu_{t}+\Delta u=|u|^{\frac{4}{d}}u,\\
u(0)=u_{0}.
\end{cases}
\end{equation}
 The local well-posedness theory for all $L_{x}^{2}$ initial data is well established and now classical, see for example \cite{cazenave2003semilinear,cazenave1990cauchy,tao2006nonlinear}.  The global well-posedness for general $L_{x}^{2}$ data is highly nontrivial, but finally solved by by Dodson \cite{dodson2012global,Dodson1,Dodson2}. In particular, he proved
 \begin{thm}\label{thm: dod}
 For all $L_{x}^{2}$ initial data, the solution $u$ to \eqref{eq: mnls} is global and one has the following spacetime bound
 \begin{equation}\label{eq: sca}
 \|u\|_{L_{t}^{q}L_{x}^{r}(\mathbb{R}\times \mathbb{R}^d)}<\infty,
 \end{equation}
 where $(q,r)$ is some admissible Strichartz pair. In particular, \eqref{eq: sca} implies there exists some $u^{+}\in L_{x}^{2}$, so that 
 \begin{equation}
\| u(t)-e^{it\Delta}u^{+}\|_{L^2_x(\mathbb{R}^d)} \xrightarrow{t\rightarrow \infty} 0.
 \end{equation}
 \end{thm}
 
 It should be noted, if one assumes the initial data is Schwarz, (or smooth and localized in certain sense), one can also establish the scattering behavior by pseudo-conformal symmetry. This is very different from energy critical model.
 
The study of well-posedness for stochastic nonlinear Schr\"odinger equation has been initiated in \cite{dBD99}. Many refinements are made, well-posedness of subcritical nature to critical nature has been explored, see \cite{BRZ14,BRZ16,dBD03,fan2018global,fan2019wong,Hor,DZhang} and the reference therein.  In short summary, one can say given an a-priori control of the critical quantity for the stochastic model, then the deterministic scattering result implies a non-blow up result for the stochastic model.

However, though many results do cover the global existence, long time dynamic are far less understood. In \cite{FX}, decay estimate was established for a linear model with small noise. In \cite{HRZ}, a scattering result was established by assuming the noise is of finite quadratic variation in time. As a comparison, for our main result Theorem \ref{thm: main}, the noise is far from be of finite quadratic variation\footnote{We treat rougher data  ($L_{x}^{2}$ regularity) compared to \cite{HRZ}, but this is not the point of current article, the focus here is about the time decay of noise.}, but since the noise is decaying, though it is not small for $t\lesssim 1$, it is also not really a large noise result. We think it will of great interest if one can recover the result of Theorem \ref{thm: main} for small but non-decaying noise.

We also point out, if one is looking at global spacetime bound which is of form $L_{\omega}^{\rho}L_{t}^{q}L_{x}^{r}$, or scattering in $L_{\omega}^{\rho}$, the mass critical model seems to be very specific. The current article relies on a pathwise conservation law, which is not flexible enough to be generalized to other model.  And in particular, those stochastic models in general does not preserve the energy. However, one can retreat to try to prove a.s. $L_{t}^{q}L_{x}^{r}$ and a.s scatter to a linear solution, or similar convergence in probability. Those type of results will be also very interesting.

We finally remark that the previous discussion (and the current article) are mainly for stochastic problems with a multiplicative noise. Stochastic NLS with an additive noise are also of great interest, see for example \cite{dBD03}, \cite{Oh} and the reference therein. Usually one should be able to study much rougher noise in the additive noise compared to the multiplicative noise.  We want to mention the recent important work \cite{YDeng1}, \cite{YDeng2}, which gives very powerful tool to establish well-posedness result if one wants to study very rough noise. See more reference and discussions in \cite{YDeng1}, \cite{YDeng2}. 

\subsection{Notations}
we write $A \lesssim B$ to say that there is a constant $C$ such that $A\leq CB$. We use $A \simeq B$ when $A \lesssim B \lesssim A $. Particularly, we write $A \lesssim_u B$ to express that $A\leq C(u)B$ for some constant $C(u)$ depending on $u$. Without special clarification, the implicit constant $C$  can vary from line to line. In addition, we denote $a\pm:=a\pm \epsilon$ with $0<\epsilon \ll 1$. Moreover, we use Japanese bracket $\langle x\rangle$ to denote $(1+|x|^2)^{\f{1}{2}}$ and $p^{'}$ for the dual index of $p>1$ in the sense of $\frac{1}{p^{'}}+\frac{1}{p}=1$. 

We sometimes will short $L_{\omega}^{a}L_{t}^{b}L_{x}^{c}(\Omega\times [t_{1},t_{2}]\times \mathbb{R}^{3})$ as $L_{\omega}^{a}L_{t}^{b}L_{x}^{c}([t_{1},t_{2}]\times \mathbb{R}^{3})$ or simply $L_{\omega}^{a}L_{t}^{b}L_{x}^{c}([t_{1},t_{2}])$. We abuse notion a bit and don't distinguish between $[t_{1},\infty]$ and $[t_{1},\infty)$. We will also short $L_{\omega}^{a}L_{t}^{b}L_{x}^{c}(\Omega\times [0,\infty]\times \mathbb{R}^{3})$ simply as $L_{\omega}^{a}L_{t}^{b}L_{x}^{c}$.
\subsection{Acknowledgment}
C.F. was partially supported by a Simons Travel grant and a start up grant from AMSS.
\section{Preliminary}
We start with a simple version of Burkholder inequality which is most relevant to the current article. 
\begin{lemma}
Let $B_{t}$ be the usual Brownian motion. Let $2\leq p<\infty$, and $\sigma$ be a continuous adapted process (to $B_{t}$) in $L_{x}^{p}$, then one has 
\begin{equation}
	\big\| \sup_{0\leq a\leq b\leq T}\|\int_{a}^{b}\sigma(s)dB_{s}\|_{L^p_x} \big\|_{L_{\omega}^{\rho}}\lesssim_{\rho, p}\big\|\int_{0}^{T}\|\sigma(s)\|^{2}_{L_{x}^{p}}ds\big\|_{L_{\omega}^{\frac{\rho}{2}}}^{\frac{1}{2}}.
\end{equation}
\end{lemma}
\begin{remark}
One cannot take $\rho$ be $\infty$.
\end{remark}
One may refer to \cite{B,BP,Bur73,BDG73} for more general versions involving $\gamma$-Randonifying operators, we don't pursue the details here.\vspace{3mm}

Then we state the standard decay estimate (also known as dispersive estimate) and Strichartz estimate for Schr\"odinger operator. We refer to \cite{tao2006nonlinear} for details.
\begin{lemma}[Dispersive estimate]
For the linear propagator $e^{it\Delta}$ of Schr\"odinger equation in $\R^{d}$, one has 
\begin{equation}\label{eq: dispersive}
\|e^{it\Delta}u_{0}\|_{L_{x}^{\infty}}\lesssim t^{-\frac{d}{2}}\|u_{0}\|_{L_{x}^{1}}.
\end{equation}
Moreover, by interpolation with the mass conservation, for $p\geq 2$, we have
\begin{equation}\label{eq: dispersive2}
\|e^{it\Delta}u_{0}\|_{L_{x}^{p}}\lesssim t^{-d(\frac{1}{2}-\frac{1}{p})}\|u_{0}\|_{L_{x}^{p}}.
\end{equation}
\end{lemma}

\begin{lemma}[Strichartz estimate] \label{Euclidean}
Suppose $\frac{2}{q}+\frac{d}{p}=\frac{d}{2}$, where $p,q \geq 2$ and $(q,p,d)\neq (2,\infty,2)$. We call such pair $(p,q)$ a Strichartz pair and denote $\mathcal{A}_d$ to be the set of all Strichartz pairs. Then 
\begin{equation}
    \|e^{it\Delta_{\mathbb{R}^d}}f\|_{L^q_tL^p_x(\mathbb{R}\times \mathbb{R}^d)} \lesssim \|f\|_{L^2(\mathbb{R}^d)}.
\end{equation}
Also, for Strichartz pairs $(p_1,q_1)$ and $(p_2,q_2)$,
\begin{equation}
    \big\|\int_0^t e^{i(t-s)\Delta}F(s)ds\big\|_{L^{q_1}_tL^{p_1}_x(\mathbb{R}\times \mathbb{R}^d)} \lesssim \|F(s)\|_{L^{q_{2}^{'}}_{t} L^{p_{2}^{'}}_x(\mathbb{R}\times \mathbb{R}^d)}.
\end{equation}
\end{lemma}
At last, we recall the conventional notations for Strichartz norm $S(I)$ and dual Strichartz norm $N(I)$ ($I$ is a time interval) as follows. 
\begin{definition}
When $d\geq 3$, 
\begin{equation}
    \|u\|_{S(I)}=\sup_{(p,q)\in \mathcal{A}_d} \|u\|_{L^q_tL^p_x(I\times \mathbb{R}^d)}
\end{equation}
and
\begin{equation}
    \|u\|_{N(I)}=\inf_{(p,q)\in \mathcal{A}_d} \|u\|_{L^{q^{'}}_{t}L^{p^{'}}_{x}(I\times \mathbb{R}^d)}.
\end{equation}
\end{definition}
Then one can rewrite the Strichartz estimate in terms of Strichartz norm $S(I)$ and dual Strichartz norm $N(I)$.

\section{Overview of the proof of Theorem \ref{thm: main}}
\subsection{Quick review}\label{sub: qr}
To start, we briefly recall the crucial a-priori estimates in \cite{fan2018global}, which are used to prove the global existence with the help of mass conservation law. \footnote{\cite{fan2018global} was written for the 1d model, but can be easily generalized to high dimensions.}Let $u$ be a solution to \eqref{eq: snls} for with initial data $u_{0}\in L^{\infty}_{\omega}L_{x}^{2}$, one has for any $1 \leq \rho<\infty$, 
\begin{equation}\label{eq: ap1}
\|u\|_{L_{\omega}^{\rho}L_{t}^{\alpha}L_{x}^{\beta}([0,1]\times \mathbb{R}^{3})}\lesssim 1.
\end{equation}
Here $L_{t}^{\alpha}L_{x}^{\beta}$ is a Strichartz-type norm. Note that since \eqref{eq: ap1} is an estimate within unit time scale, the value of $\gamma$ does not matter at all.

We quickly review the proof here, for notation convenience, we take $\gamma=0$.
Write down the Duhamel Formula of \eqref{eq: snls} for $t$ in $[0,1]$ as 
\begin{equation}\label{eq: duhamel1}
u(t,x)=S(t)u_{0}-i\int_{0}^{t}S(t-s)N(u(s))ds-i\int_{0}^{t}S(t-s)(V(x)u(s))dB_{s}-\frac{1}{2}\int_{0}^{t}S(t-s)(V^{2}u(s))ds,
\end{equation}
where $S(t)$ is the free Schr\"odinger propagator. And note for any $t\in [a,b]\subset [0,1]$, one has 
\begin{equation}\label{eq: duhamel22}
\begin{aligned}
u(t,x)=&S(t-a)u(a)-i\int_{a}^{t}S(t-s)N(u(s))ds-\\
&i\int_{a}^{t}S(t-s)(V(x)u(s))dB_{s}-\frac{1}{2}\int_{0}^{t}S(t-s)(V^{2}u(s))ds.
\end{aligned}
\end{equation}
Schematically, we will view the term $-i\int_{0}^{t}S(t-s)(V(x)u(s))dB_{s}-\frac{1}{2}\int_{0}^{t}S(t-s)(V^{2}u(s))ds$ as a perturbation to the deterministic NLS.  More precisely, we will consider a maximal version of such a term by letting 
\begin{equation}\label{eq: mstar }
M^{*}(t):=\sup_{0\leq a\leq b\leq t}\big\|i\int_{a}^{b}S(t-s)(Vu(s))dB_{s}+\frac{1}{2}\int_{a}^{b}S(t-s)(V^{2}u(s))ds\big\|_{L_{x}^{\beta}}.
\end{equation}

For concreteness and simplicity, one can take, for example, $(\alpha,\beta):=(\frac{14}{3}, \frac{14}{5})$ or any  $(4+,3-)$ which are admissible Strichartz pairs. We remark, $(\frac{14}{3}, \frac{14}{5})$ and $(4,3)$ are two endpoints for our technique, and the latter cannot be attained.\footnote{That being said, since one has an extra $L_\omega^{\infty}L_{t}^{\infty}L_{x}^{2}$ by pathwise mass conservation law, the endpoint $\frac{14}{3}$ can be removed via interpolation, and one cover all $(\alpha,\beta)$ admissible pair as far as $\alpha>4$.}

The proof of \eqref{eq: ap1} will be reduced to the following two estimates. We fix $\rho$ first.
\begin{itemize}
\item By a pathwise mass conservation law\footnote{Because of this, even to prove \eqref{eq: ap1} for a single $\rho$, we still need the initial data be in $L_{\omega}^{\infty}L_{x}^{2}$.}, we have a-priori estimate $\|u\|_{L_{x}^{2}}\lesssim 1$, and we can use a deterministic modified stability argument\footnote{In this article, we will do an improved version later, so we neglect the details for the moment.} based on NLS, see Proposition 2.6 in \cite{fan2019wong}. And one conclude in a pathwise sense,
\begin{equation}\label{eq: deterministics}
\|u\|_{L_{t}^{\alpha}L_{x}^{\beta}([a,b]\times \mathbb{R}^3)}\lesssim \|M^{*}(t)\|_{L_{t}^{\alpha}([a,b])}+1.
\end{equation}
And this step one necessarily need $\alpha \leq \frac{14}{3}$.
\item By a-priori control $\|u\|_{{L_{\omega}^{\rho}L_{t}^{\infty}}L_{x}^{2}}\lesssim 1$ and Burkholder inequality, we will be able to conclude 
\begin{equation}\label{eq: slocal}
\|M^{*}(t)\|_{L_{\omega}^{\rho}L_{t}^{\frac{14}{3}}([0,1])}\lesssim \|u(t)\|_{L_{\omega}^{\rho}L_{t}^{\infty}L_{x}^{2}([0,1]\times \mathbb{R}^3)}\lesssim \|u_{0}\|_{L^{\infty}_{\omega}L_{x}^{2}}.
\end{equation}
This step requires $\alpha> 4$ (which is equivalent  to $\beta<3$). To see why  the number $3$ is special in $\mathbb{R}^{3}$. Recall one has $\|e^{it\Delta}\|_{L_{x}^{\frac{3}{2}}\rightarrow L_{x}^{3}}\leq t^{-\frac{1}{2}}$ and $\frac{1}{t}$ is not integrable at $0$.
\end{itemize}
\subsection{Warm up}
Now we are ready to go back the proof of Theorem \ref{thm: main}.

Recall in the scheme discussed in Subsection \eqref{sub: qr}, we are not exploring any (potential) decay property of $u$, since we are basing on mass \textbf{conservation} law.
Of course, since our noise decays at a rate $\langle t \rangle^{-\gamma}$, one is indeed working on $\langle t \rangle^{-\gamma}u$ rather than $u$.  
We first remark, the above scheme\footnote{A slight modification of Proposition 2.6 in \cite{fan2019wong} will still be needed, besides natural generalization to $3$. More precisely, one needs to observe, when $\gamma$ is not very small, $\langle t \rangle^{-2\gamma}V^{2}u$ will lies in $L_{t}^{2}L_{x}^{\frac{6}{5}}$. Because$V$ is localized and $\|u\|_{L_{x}^{2}}$ is conserved.} is enough to handle $\gamma>\gamma_{0}$ for some $\gamma_{0}<\frac{1}{2}$. And in that case, one can indeed get upgrade space time bound and scattering in Theorem \ref{thm: main} to all $L_{\omega}^{\rho}, \rho<\infty$.  We leave it to interested readers.
This in particular covers,  for example, $\gamma=\frac{1}{2}$,  which is of infinite quadratic variation in time. It is because our previous scheme already explore some dispersive property which compensates the fact $\int \langle t \rangle^{-2*(\frac{1}{2})}=\infty$.

However, to handle all $\gamma>0$, one has to improve the above scheme in several different ways. And in particular, extra efforts need to be paid to upgrade the space time bound to scattering dynamic.

In the rest of the article, we fix $\gamma=\epsilon_{0}>0$, (We only consider $\epsilon_{0}$ small. Since the smaller the $\epsilon_{0}$, the harder the problem).  We also fix $\|u_{0}\|_{L_{\omega}^{\infty}L_{x}^{2}}=M$.

 We will fix $\epsilon_{0}, M$ throughout and we don't track the reliance on those two parameters $\epsilon_{0}, M$ in the rest of the article.

And we make \eqref{eq: duhamel1} more precise
\begin{equation}\label{eq: duhamel2}
\begin{aligned}
u(t,x)=&S(t)u_{0}-i\int_{0}^{t}S(t-s)N(u(s))ds-i\int_{0}^{t}S(t-s)(V(x)\langle s \rangle^{-\epsilon_{0}}u(s))dB_{s}\\
&-\frac{1}{2}\int_{0}^{t}S(t-s)(\langle s \rangle^{-2\epsilon_{0}}V^{2}u(s))ds.
\end{aligned}
\end{equation}

\subsection{Step 1: Bootstrap and spacetime bound}
The idea is if one did prove $u$ to \eqref{eq: snls} satisfy the spacetime bound for some Strichartz pair $(\alpha,\beta)$, then $\|u\|_{L_{x}^\beta}$ decays as $\frac{1}{\alpha-}$ in the average sense. It is possible to explore this information via a bootstrap scheme\footnote{It is also referred as continuity argument.} since one has an extra time decay $\langle t \rangle^{-\epsilon_{0}}$. It is tempting to fulfill this idea using the pair $(2,6)$ or $L_{t}^{2}L_{x}^{6}$ norm, which is one endpoint of classical Strichartz estimate. However, for technical reason, $(4,3)$ will be an endpoint for us, and $(4,3)$ itself is forbidden in our analysis.

Now fix $(\alpha_{0},\beta_{0})$ an admissible Strichartz pair\footnote{In the first reading, it is suggested that the reader just formally take $(\alpha_{0}, \beta_{0})=(4,3)$ and neglect all log type divergence.}, and we choose $\beta_{0}$ close to $3$ enough so that
\begin{equation}\label{eq: close}
 3(\frac{1}{2}-\frac{1}{\beta})=\frac{1}{2}-\eta_{0},
\end{equation}
where $\eta_{0}\ll \epsilon_{0}$.
And in particular
\begin{equation}
\|e^{it\Delta}\|_{L^{\beta_{0}^{'}}_{x}\rightarrow L^{\beta_0}}\lesssim t^{-\frac{1}{2}+\eta_{0}}.
\end{equation}

And we will have
\begin{lemma}\label{lem: boot}
For all $T>s>0$, there exists some $C_{1,s}$, depending on $(\|u_{0}\|_{L^{\infty}_{\omega}L_{x}^{2}}, \alpha_{0}, \epsilon_{0})$ and $s$, and $C_{2}$ depending on $(\|u_{0}\|_{L^{\infty}_{\omega}L_{x}^{2}}, \alpha_{0}, \epsilon_{0})$, such that
\begin{equation}\label{eq: boot}
\begin{aligned}
&\|u\|_{L^{\alpha_{0}}_{\omega}L_{t}^{\alpha_{0}}L_{x}^{\beta_{0}}([0,T]\times \mathbb{R}^{3})}\\
\lesssim &C_{1,s}+\langle s \rangle^{-\frac{\epsilon_{0}}{2}}C_{2}\|u\|_{L^{\alpha_{0}}_{\omega}L_{t}^{\alpha_{0}}L_{x}^{\beta_{0}}([0,T])}.
\end{aligned}
\end{equation}
Note both $C_{1}, C_{2}$ does not depend on $T$.
\end{lemma}
 Note when $s\leq T\leq 1$, Lemma \ref{lem: boot} is just local theory\footnote{However, one feature of the method is that even local theory is not completely perturbative and allows very nonlinear dynamics.}, which follows from \cite{fan2018global}, see also \cite{DZhang}. 
 
 And the spacetime bound \eqref{eq: spacetimebound3d} will follow from Lemma \ref{lem: boot} via standard continuity argument.
 
 We will present the proof of Lemma \ref{lem: boot} in Section \ref{sec: boot}. But we explain one extra key idea here.
Even one applies a bootstrap strategy, a priori $u$ (formally) decays at most like $t^{-\frac{1}{4-}}$, and in particular $\langle t \rangle^{-2\epsilon_{0}}V^{2}u(t)$ is far from in the Dual of Strichartz space\footnote{One needs $\langle t \rangle^{-2\epsilon_{0}}V^{2}u(t)$ to be integrable in $L_{t}^{2}$ at least.}.  Thus, one attacks the bootstrap Lemma based on Duhamel formula \eqref{eq: duhamel2} and treat $-\frac{1}{2}\int_{0}^{t}S(t-s)(\langle s \rangle^{-2\epsilon_{0}}V^{2}u(s))ds$ in a pure perturbative way as in \cite{fan2018global}, this cannot work.

The observation is that this term may not need to be handled in a perturbative way.  We may not perturb around NLS but equations as follows (damped NLS),
\begin{equation}\label{eq: NLSD}
iw_{t}+\Delta w=|w|^{\frac{4}{d}}-i\langle t \rangle^{-2\epsilon_{0}}V^{2}w.
\end{equation}
Note that the sign before $V^{2}$ is good\footnote{Here it is important $V$ is real.} in the sense this is a dissipative term or a damping term.

Let $H(t,s)$ be the linear propagator \footnote{Since this equation is not time translation invariant, we cannot write as $H(t-s)$.}from $s$ to $t$,
\begin{equation}
iw_{t}+\Delta w=-i\langle t \rangle^{-2\epsilon_{0}}V^{2}w.
\end{equation}
i.e. if $w(s)=f$, then we will have $H(t,s)f:=w(t)$.
 
We rewrite the Duhamel formula \eqref{eq: duhamel2} as 
\begin{equation}\label{eq: duhamel1s}
u(t,x)=H(t)u_{0}-i\int_{0}^{t}H(t,s)(|u|^{\frac{4}{d}}u)ds-i\int_{0}^{t}H(t,s)(V(x)\langle s \rangle^{-\epsilon_{0}}u(s))dB_{s}.
\end{equation}
And note for any $t\in [a,b]\subset [0,1]$, one has 
\begin{equation}\label{eq: duhamel2s}
\begin{aligned}
u(t,x)=&H(t,a)u(a)-i\int_{a}^{t}H(t,s)(|u|^{\frac{4}{d}}u)ds-\\
&i\int_{a}^{t}H(t,s)(V(x)\langle s \rangle^{-\epsilon_{0}}u(s))dB_{s}
\end{aligned}
\end{equation}
and based on this, we will study the following quantity rather than $M^{*}$ in \eqref{eq: mstar },
\begin{equation}\label{eq: m0}
M_{0}(t):=\sup_{0\leq a\leq b\leq t}\big\|i\int_{a}^{b}H(t,s)(\langle s \rangle^{-\epsilon_{0}}Vu(s))dB_{s}\|_{L_x^{\beta_0}}.
\end{equation}
And we will do a improved modified stability based on damped NLS, see Proposition \ref{prop: modistab}. On one hand, Proposition \ref{prop: modistab} generalized the result of Proposition 2.6 \cite{fan2019wong}, on the other hand, it is an improved version which gives more detailed characterization of the solution which will play a crucial role in the next two steps.

See Section \ref{sec: boot} for proof of Lemma \ref{lem: boot}.

\subsection{Step 2: Decompose the solution via improved modified stability }
Recall step 1 gives us a space time global bound $L_{\omega}^{\alpha_{0}}L_{t}^{\alpha_{0}}L_{x}^{\beta_{0}}([0,\infty)\times \mathbb{R}^{3})$.

Now, rewrite the Duhamel formula of \eqref{eq: snls} based on usual NLS again, also recall we have fixed $\gamma=\epsilon_{0}$.
\begin{equation}\label{eq: duf3}
\aligned
&u(t,x)=S(t)u_{0}-i\int_{0}^{t}S(t-s)(|u|^{\frac{4}{3}}u(s))ds-i\int_{0}^{t}S(t-s)(\langle s \rangle^{-\epsilon_{0}}V(x)u(s))dB_{s}\\
&-\frac{1}{2}\int_{0}^{t}S(t-s)(\langle s \rangle^{-2\epsilon_{0}}V^{2}u(s))ds.
\endaligned
\end{equation}

The spacetime abound $L_{\omega}^{\alpha_{0}}L_{t}^{\alpha_{0}}L_{x}^{\beta_{0}}([0,\infty)\times \mathbb{R}^{3})$ is enough to handle the nonlinear part $|u|^{\frac{4}{3}}u$ in the study of the scattering dynamic.

However, heuristically a typical  $L_{t}^{\alpha_{0}}L_{x}^{\beta_{0}}$ function decays at most like $\frac{1}{\alpha_{0}}\sim \langle t \rangle^{-\frac{1}{4}+}$.  Even with the extra decay $\langle t \rangle^{-\epsilon_{0}}$, it is far from be integrable in $L_{t}^{2}$,  when the $\epsilon_{0}$ is small. This fact stops us to handle the linear stochastic part and linear damping part in \eqref{eq: duf3}. 

The key idea here now is to consider a new maximal type quantity
\begin{equation}\label{eq: M1}
M_1^{*}(t)=\sup_{0\leq r_{1}\leq r_{2}\leq t}\big\|\int_{r_{1}}^{r_{2}}H(t,s)(\langle s \rangle^{-\epsilon_{0}}Vu)dB_{s}\big\|_{L_{x}^{\beta_1}}.
\end{equation}
with $\beta_{1}$ chosen to close to $3$ enough such that
\begin{lemma}\label{lem: m1}
Let $u$ be the solution to \eqref{eq: snls} with $\|u\|_{L_{\omega}^{\alpha_{0}}L_{t}^{\alpha_{0}}L_{x}^{\beta_{0}}([0,\infty))}<\infty$. There exists $\beta_{1}$ close to $3$ enough, such that for some $\alpha_{1}$ with
\begin{equation}\label{eq: close2}
\frac{1}{\alpha_{1}}\geq \frac{1}{\alpha_{0}}+\frac{\epsilon_{0}}{10},
\end{equation}
one has 
\begin{equation}\label{eq: estimateform1}
\|M_{1}^{*}(t)\|_{L_{\omega}^{\alpha_{1}}L_{t}^{\alpha_{1}}([0,\infty))}\lesssim
\|u\|_{L_{\omega}^{\alpha_{0}}L_{t}^{\alpha_{0}}L_{x}^{\beta_{0}}([0,\infty))}.
\end{equation}
\end{lemma}
\begin{remark}
One crucial point here, we now \textbf{don't need} $(\alpha_{1}, \beta_{1})$ to be Strichartz admissible anymore.
\end{remark}
And we then take the full advantage of our improved modified stability, Proposition \ref{prop: modistab}, to prove
\begin{lemma}\label{lem: splitintotwo}
Let $u$ be the solution to \eqref{eq: snls} with $\|M_0^{*}(0,t)\|_{ L_{\omega}^{\alpha_0}L_t^{\alpha_0}}<\infty$,  \\and $\|M_1^{*}(0,t)\|_{ L_{\omega}^{\alpha_1}L_t^{\alpha_1}}<\infty$, then one can write $u=u_1+u_2$ such that  $u_1 \in L_{\omega}^{2}L_t^{2}L_x^{6}$ and $u_2 \in L_{\omega}^{\alpha_1}L_t^{\alpha_1}L_x^{{\beta}_1}$ with estimates
\begin{equation}\label{eq: firstestiamteforu1}
\|u_{1}\|_{L_{\omega}^{2}L_t^{2}L_x^{6}}\lesssim \|M^{*}_{0}(0,t)\|^{\frac{\alpha_0}{2}}_{ L_{\omega}^{\alpha_0}L_t^{\alpha_0}}+1,
\end{equation}
and
\begin{equation}\label{eq: fisrtestimateforu2}
\|u_{2}\|_{L_{\omega}^{\alpha_{1}}L_t^{\alpha_{1}}L_x^{\beta_{1}}}\lesssim \|M^{*}_{1}(t)\|_{ L_{\omega}^{\alpha_1}L_t^{\alpha_1}}.
\end{equation}
In particular,
\begin{equation}\label{eq: firstimprove}
 |V^{\frac{1}{2}}u|\in L_{\omega}^{\alpha_1}L_t^{\alpha_1}L^2_x.
\end{equation}
\end{lemma}
Note that via  $|V^{\frac{1}{2}}u|\in L_{\omega}^{\alpha_1}L_t^{\alpha_1}L^2_x$, and the fact $\frac{1}{\alpha_{1}}\geq \frac{1}{\alpha_{0}}+\frac{\epsilon_{0}}{10}$, we  have already improve the decay of $u$ and this allows us to treat smaller $\epsilon_{0}$.

We will prove Lemma \ref{lem: m1}, \ref{lem: splitintotwo} in Section \ref{sec: fisrtsplit}.
\subsection{Step 3: Iteration process to give enough decay for scattering}
To handle all $\epsilon_{0}>0$, we need to iterate the control of step 2 until $\alpha_{1}$ goes all the way to $2$, (or close to 2 enough depending on the value of $\epsilon_{0}$).

We will consider for $n\geq 1$,
\begin{equation}\label{eq: mn}
M_{n}^{*}(t):=\sup_{0\leq r_{1}\leq r_{2}\leq t}\big\|\int_{r_{1}}^{r_{2}}H(t,s)(\langle s \rangle^{-\epsilon_{0}}Vu)dB_{s}\big\|_{L_{x}^{\beta_1}}.
\end{equation}
where $\beta_{n}=\beta_{1}, \forall n.$ And prove iteratively
\begin{lemma}\label{lem: mn}
$\forall n\geq 1, \alpha_{n}\geq 2$. Let $u$ be the solution to \eqref{eq: snls}  with $|V^{\frac{1}{2}}u|\in L_{\omega}^{\alpha_{n}}L_t^{\alpha_{n}}L^2_x$, then there exists $\alpha_{n+1}$ such that
\begin{equation}\label{eq: iteformula}
\frac{1}{\alpha_{n+1}}\geq \frac{1}{\alpha_{n}}+\frac{\epsilon_{0}}{10},
\end{equation}
and 
\begin{equation}
    \|M_n^{\ast}(t)\|_{L^{\alpha_{n+1}}_{\omega}L_t^{\alpha_{n+1}}} \lesssim \|V^{\frac{1}{2}}u\|_{L_{\omega}^{\alpha_{n}}L_t^{\alpha_{n}}L^2_x} . 
\end{equation}
\end{lemma}
\begin{lemma}\label{lem: splitagain}
Let $u$ be the solution to \eqref{eq: snls} with $\|M_0^{*}(0,t)\|_{ L_{\omega}^{\alpha_0}L_t^{\alpha_0}}<\infty$. If $\|M_n^{*}(t)\|_{ L_{\omega}^{\alpha_1}L_t^{\alpha_n}}<\infty$, then one can write $u=u_1+u_2$ such that  $u_1 \in L_{\omega}^{2}L_t^{2}L_x^{6}$ and $u_2 \in L_{\omega}^{\alpha_{n}}L_t^{\alpha_{n}}L_x^{{\beta}_1}$ with estimates
\begin{equation}\label{eq: iteestiamteforu1}
\|u_{1}\|_{L_{\omega}^{2}L_t^{2}L_x^{6}}\lesssim \|M_{0}^{*}(0,t)\|^{\frac{\alpha_0}{2}}_{ L_{\omega}^{\alpha_0}L_t^{\alpha_0}}+1,
\end{equation}
and
\begin{equation}\label{eq: iteestimateforu2}
\|u_{2}\|_{L_{\omega}^{\alpha_{n}}L_t^{\alpha_{n}}L_x^{\beta_{1}}}\lesssim \|M_{n}^{*}(t)\|_{ L_{\omega}^{\alpha_n}L_t^{\alpha_n}}
\end{equation}
In particular,
\begin{equation}\label{eq: iteimprove}
 |V^{\frac{1}{2}}u|\in L_{\omega}^{\alpha_{n}}L_t^{\alpha_{n}}L^2_x.
\end{equation}
\end{lemma}
And we will stop iteration at some $n_{0}$, such that $\alpha_{n_{0}}\leq 2$.  Thanks to \eqref{eq: iteformula}, the iteration will stop at finite steps and $n_{0}<\infty$. And we may (from the proof) take $\alpha_{n_{0}}=2$

One will see the proof of Lemma \ref{lem: mn}, \ref{lem: splitagain} is essentially same (almost line by line same) as the proof of Lemma \ref{lem: m1}, \ref{lem: splitintotwo}.

To summarize, combining \eqref{eq: spacetimebound3d} will follow from Lemma \ref{lem: boot}, we have
\begin{lemma}\label{lem: summary}
Let $u$ solves \eqref{eq: snls}, we have 
\begin{equation}
 \|u\|_{L_{\omega}^{\alpha_{0}}L_{t}^{\alpha_{0}}L_{x}^{\beta_0}([0,\infty)\times \mathbb{R}^{3})}\lesssim_{\|u_{0}\|_{L_{\omega}^{\infty}L_{x}^{2}}, \alpha_{0},\epsilon_0} 1,
\end{equation}
and 
\begin{equation}\label{eq: fffinal}
 \||V|^{\frac{1}{2}}u\|_{L_{\omega}^{2}L_{t}^{2}L_{x}^{2}([0,\infty)\times \mathbb{R}^{3})}\lesssim_{\|u_{0}\|_{L_{\omega}^{\infty}L_{x}^{2}}, \alpha_{0},\epsilon_0} 1.
\end{equation}
\end{lemma}

We finally remark, by exploring the martingale structure in the stochastic part in \eqref{eq: duf3} and observe other parts are handled via integration in time, one can upgrade convergence in $L_{\omega}^{\rho^{*}}$ in Theorem \ref{thm: main} into almost sure convergence.

We will quickly go over the proof of Lemma \ref{lem: mn}, \ref{lem: splitagain} and explain how Lemma \ref{lem: summary} gives \eqref{eq: scattering3d} and \eqref{eq: almostsure} in Section \ref{sec: conclude}.

\section{Proof of Lemma \ref{lem: boot}}\label{sec: boot}
Recall in \eqref{eq: snls}, we fix $\gamma=\epsilon_{0}>0$. We also fix $\|u_{0}\|_{L_{\omega}^{\infty}L_{x}^{2}}=M$.  All constants in the analysis may (implicitly) rely on those two parameters.

Recall also  Strichartz pair $(\alpha_0,\beta_0)=(4+,3-)$ is chosen so that
\eqref{eq: close} holds.

To start the proof of Lemma \ref{lem: boot}, we first present all the needed property for Damped NLS in dimension 3.
\begin{equation}
iw_{t}-\Delta w=|w|^{\frac{4}{3}}w-i\langle t \rangle^{-2\epsilon_0} V^{2}w.
\end{equation}
\subsection{Properties for Damped NLS }
We start with the linear dispersive estimate. Recall $H(t,s)$ be the linear propagator to 
\begin{equation}\label{eq: dls}
iw_{t}-\Delta w+i\langle t \rangle^{-2\epsilon_0} V^{2}w=0.
\end{equation}
Based on the crucial observation in \cite{JSS}, the term $i\langle t \rangle^{-2\epsilon_0} V^{2}w$ can be treated in a perturbative way (see also \cite{Schlagsurvey}).
\begin{lemma}\label{lem: dlsdis}
For $t>s$, one has dispersive estimate (uniform in $t,s$) as follows,
\begin{equation}
\|H(t,s)f\|_{L_{x}^{\infty}}\lesssim \|f\|_{L_{x}^{1}}(t-s)^{-\frac{3}{2}}.
\end{equation}
\end{lemma}
Note that a direct energy estimate gives 
\begin{equation}
\|H(t,s)f\|_{L_{x}^{2}}\leq \|f\|_{L_{x}^{2}}.
\end{equation}
Thus, one can apply the classical Riesz-Thorin interpolation to derive
\begin{equation}\label{eq: esdisdls}
\|H(t,s)f\|_{L_{x}^{p}}\lesssim \|f\|_{L_{x}^{p'}}(t-s)^{-3(\frac{1}{2}-\frac{1}{p})}.
\end{equation}
\
Next, we present the associated Strichartz estimates for the free solution and Duhamel term for \eqref{eq: dls},
\begin{lemma}
We have the following Strichartz type estimates
\begin{equation}\label{eq: dlsstri}
\|H(t,a)f\|_{S(a,\infty)}\lesssim \|f\|_{L_{x}^{2}(\mathbb{R}^{3})},
\end{equation}
and
\begin{equation}\label{eq: dlsstri2}
\|\int_{a}^{t}H(t,s)F(s)ds\|_{S(a,\infty)}\lesssim \|F\|_{N(a,\infty)},
\end{equation}
where we recall $S$ is the Strichartz norm and $N$ is the dual Strichartz norm.
\end{lemma}
\begin{proof}
The proof does not rely on \eqref{eq: dls}. The proof can be obtained by certain energy estimates and Strichartz estimates for linear Schr\"odinger equation. Without loss of generality, we consider $a=0$. We denote $\tilde{V}=\langle t \rangle^{-\epsilon_0} V$.

Let $w$ solves \eqref{eq: dls} with initial condition $w(0)=f$, i.e. $w=H(t,0)f$. 
We first observe that
\begin{equation}
    \frac{d}{dt}\int |w|^2dx=-\int \tilde{V}^2 |w|^2 dx,
\end{equation}
by Fundamental Theorem of Calculus, we have 
\begin{equation}
    \int_0^{+\infty} |\tilde{V}|^2|w|^2dxdt \lesssim \|w_0\|^2_{L^2_x}.
\end{equation}
Moreover, noticing $V$ is localized and applying the H\"older, we have
\begin{equation}
    \|\tilde{V}^2w\|_{L_x^{\frac{6}{5}}}\lesssim \|\tilde{V} w\|_{L_x^2}
\end{equation}
then we see  
\begin{equation}
    \| \tilde{V}^2w \|_{L^2_tL_x^{\frac{6}{5}}} \lesssim \|f\|_{L^2_x}.
\end{equation}
By Duhamel Formula based on Schr\"odinger equation, we write $w$ as
\begin{equation}
w(t,x)=S(t)w(0)-i\int_{0}^{t}S(t-s)(-i\tilde{V}^2w)ds.
\end{equation}
Now one may apply the Strichartz estimate to obtain \eqref{eq: dlsstri}.

Then we move on to the control of the Duhamel term. Consider
\begin{equation}
iw_{t}+\Delta w=-i\tilde{V}^{2}w+F, \quad w(0,x)=0.
\end{equation}
Direct calculations give
\begin{equation}
    \frac{d}{dt}\int|w|^2=-\int \tilde{V}^2w^2+O(\int |F|\cdot|w|),
\end{equation}
which further implies the following important a-priori estimate
\begin{equation}
\int \tilde{V}^{2}w^{2}\leq \int_{t,x} |F||w|dtdx \leq \|F\|_{N}\|w\|_{S}.
\end{equation}
Plug this into the Duhamel Formula based on usual Schr\"odinger equation, this gives
\begin{equation}
\|w\|_{S}\leq C(\|F\|_{N}+\|w\|_{S}^{\frac{1}{2}}\|F\|_{N}^{\frac{1}{2}}).
\end{equation}
This implies \footnote{We note that rigorous proof indeed involves a bootstrap argument.}
\begin{equation}
\|w\|_{S(0,\infty)}\leq \|F\|_{N(0,\infty)}.
\end{equation}
Similarly, we derive
\begin{equation}\label{eq: 11}
\big\|\int_{a}^{t}H(t,s)F(s)ds\big\|_{S(a,b)}\leq \|F\|_{N(a,b)}.
\end{equation}

\end{proof}
Next, we observe that Dodson's global well-posedness result \cite{dodson2012global} can be used as a blackbox to derive the associated results for 
\begin{equation}\label{eq: dnls}
\begin{cases}
iw_{t}-\Delta w+i\langle t \rangle^{-2\epsilon_0}V^{2}w=|w|^{\frac{4}{3}}w,\\
w(0)=w_{0}.
\end{cases}
\end{equation}
We have,
\begin{lemma}\label{lem: gwpdnls}
For $w$ to \eqref{eq: dnls} with initial data $w_{0}\in L_{x}^{2}$, the solution is global forward in time, and satisfies 
\begin{equation}\label{eq: gbdnls}
\|w\|_{L_{t}^{\alpha}L_{x}^{\beta}((0,+\infty)\times \mathbb{R}^3)}\leq \|w\|_{L_{t}^{2}L_{x}^{6}\cap L_{t}^{\infty}L_{x}^{2}}\lesssim_{\|w_{0}\|_{L_{x}^{2}}} 1,
\end{equation}
where $(\alpha,\beta)$ is Strichartz admissible pair.
\end{lemma}
\begin{proof}
First, we note that $L_{t}^{\alpha}L_{x}^{\beta}$ is a Strichartz-type norm and if one removes the term $i\langle t \rangle^{-2\epsilon_0}V^2w$, the damped NLS model \eqref{eq: dnls} will become the mass-critical NLS. And Dodson \cite{dodson2012global} proved the global controls for the Strichartz norm of the solutions. 

We denote $\tilde{V}=\langle t \rangle^{-\epsilon_0} V$. Similar to the proof of Lemma 4.3, we have
\begin{equation}
    \frac{d}{dt}\int |w|^2dx=-\int \tilde{V}^2 |w|^2 dx,
\end{equation}
then by Fundamental Theorem of Calculus, we have
\begin{equation}
    \int_0^{+\infty} |\tilde{V}|^2|w|^2dxdt \lesssim \|w_0\|^2_{L^2_x}.
\end{equation}
Moreover, since $V$ is localized we obtain,
\begin{equation}
    \|\tilde{V}^2w\|_{L_x^{\frac{6}{5}}}\lesssim \|\tilde{V}w\|_{L_x^2}.
\end{equation}
Putting the above two estimates together, we have the global control for $i\tilde{V}^2w$ as follows
\begin{equation}
    \|i\tilde{V}^2w\|_{L_t^2L_x^{\frac{6}{5}}((0,+\infty)\times \mathbb{R}^3)}<\infty.
\end{equation}
We emphasize that the above global estimate relies on the fact that $i\langle t \rangle^{-2\epsilon_0}V^{2}w$ is a damping term in the damped NLS \eqref{eq: dnls}. This allows us to treat the term $i\tilde{V}^2w$ in a perturbative way. Then we split the whole time interval $(0,+\infty)$ into finite number of disjoint subintervals $\{I_i\}_i$ such that on each subinterval $I_i$, the error term $i\tilde{V}^2w$ is small enough in the following sense
\begin{equation}
    \|i\tilde{V}^2w\|_{L_t^2L_x^{\frac{6}{5}}(I_i \times \mathbb{R}^3)}<\epsilon.
\end{equation}
Then on every such time subinterval, we apply the  perturbation theory for the mass-critical NLS (see \cite{dodson2012global}) and sum the norms $L_{t}^{\alpha}L_{x}^{\beta}$ associated with the subintervals up to obtain \eqref{eq: gbdnls}. 
\end{proof}
We note that Lemma 4.2 implies the local theory for the damped NLS \eqref{eq: dls} by standard arguments. Furthermore, with the help of Lemma \ref{lem: gwpdnls}, we prove the following stability result.
\begin{lemma}\label{lem: stadnls}
Let $w$ solves \eqref{eq: dnls} in interval [a,b], with $w(a)=w_{0}$, and $\|w_{0}\|_{L_{x}^{2}}=M$, then there exists $\epsilon=\epsilon_{M}$, so that for all $\tilde{w}_{0}, e$ with
\begin{equation}
\|w_{0}-\tilde{w}_{0}\|<\epsilon, \|e\|_{L_{t}^{1}L_{x}^{2}([a,b]\times \mathbb{R}^3)}<\epsilon,
\end{equation}
one has solution to
\begin{equation}
\begin{cases}
i\tilde{w}_t-\Delta\tilde{w}+i\langle t \rangle^{-2\epsilon_0}V^{2}\tilde{w}=|\tilde{w}|^{\frac{4}{3}}\tilde{w}+e,\\
\tilde{w}(a)=\tilde{w}_{0}
\end{cases}
\end{equation}
in $[a,b]$ with estimate
\begin{equation}
\|w-\tilde{w}\|_{L_{t}^{\alpha}L_{x}^{\beta}\cap L_{t}^{\infty}L_{x}^{2} ([a,b]\times \mathbb{R}^3)}\lesssim_{M}(\|w(a)-\tilde{w}(a)\|_{L_{x}^{2}}+\|e\|_{N([a,b])})
\end{equation}
and in particular 
\begin{equation}
\|w-\tilde{w}\|_{L_{t}^{\alpha}L_{x}^{\beta}\cap L_{t}^{\infty}L_{x}^{2} ([a,b]\times \mathbb{R}^3)}\lesssim_{M} 1,
\end{equation}
where $(\alpha,\beta)$ is Strichartz pair.
\end{lemma}
\begin{remark}
All the implicit constants in Lemma \ref{lem: stadnls} are independent of the choice of interval $[a,b]$.
\end{remark}

\begin{proof}
We note that in Lemma 4.2, we have established the Strichartz estimates for the damped Schr\"odinger equation. Also, in Lemma \ref{lem: gwpdnls}, we have established a global spacetime bound. Then the proof follows in a standard way so we omit it. See Lemma 3.10 in  \cite{colliander2008global} (energy-critical NLS) for more details. 

We note that the difference of Lemma \ref{lem: stadnls} from Lemma 3.10 in \cite{colliander2008global} is that we already have the a-priori global bound established so we do not need further assumptions. (The authors assumed $L^{10}_{t,x}$ spacetime bound in Lemma 3.10 of \cite{colliander2008global}.)  
\end{proof}
\subsection{Modified stability}
We are ready to state an analogue of Proposition 2.6 in \cite{fan2019wong} as follows,
\begin{proposition}\label{prop: modistab}
Assume $w$ solves in $[a,b]$,
\begin{equation}
w(t)=H(t,a)w(a)-i\int_{a}^{t}H(t,s)w(s)ds+\eta(t)
\end{equation}
with 
\begin{equation}
\|w(t)\|_{L_{t}^{\infty}L_{x}^{2}}\leq M,
\end{equation}
and
\begin{equation}\label{eq: isint}
	\eta(a)=0.
\end{equation}
We note $(\alpha,\beta)$ is Strichartz pair and $\frac{7}{3} \leq \alpha \leq \frac{14}{3}$. Then, there exists $\epsilon_{M}, B_{M}>0$, such that if
\begin{equation}
\|\eta\|_{L_{t}^{\alpha}L_{x}^{\beta}([a,b]\times \mathbb{R}^3)}\leq \epsilon_{M},
\end{equation}
then
\begin{equation}\label{eq: extrakey}
	\|w(t)-\eta(t)\|_{L_{t}^{2}L_{x}^{6}\cap L_{t}^{\infty}L_{x}^{2}([a,b]\times \mathbb{R}^3)}\leq \frac{B_{M}}{2},
\end{equation}
and
\begin{equation}\label{eq: okper}
\|w\|_{L_{t}^{\alpha}L_{x}^{\beta}([a,b]\times \mathbb{R}^3)}\leq \frac{B_{M}}{2}+\epsilon_{M}\leq B_{M}.
\end{equation}
\end{proposition}
\begin{remark}
It should be noted here one needs control for the Strichartz norm of $\eta$ rather than its dual Strichartz norm. We also obtain more refined bound \eqref{eq: extrakey} compared to a direct generalization of Proposition 2.6 in \cite{fan2019wong}. Estimate \eqref{eq: extrakey} is important to derive the scattering asymptotic.
\end{remark}
\begin{proof}
The proof is essentially same as Proposition 2.6 in \cite{fan2019wong}, given one already obtains Lemma \ref{lem: gwpdnls} and Lemma \ref{lem: stadnls}. We give a rather detailed sketches for the reader and in particular explain how to derive \eqref{eq: extrakey}. Before we start, we first observe $L_{t}^{2}L_{x}^{6}\cap L_{t}^{\infty}L_{x}^{2}\hookrightarrow L_{t}^{\alpha}L_{x}^{\beta}$. 

The key idea of the proof is to observe if one let
\begin{equation}
	v=w-\eta  
\end{equation}
then $v$ solves 
\begin{equation}
	\begin{cases}
	iv_{t}-\Delta v+i\langle t \rangle^{-2\epsilon_0} V^{2}v=|v+\eta|^{\frac{4}{3}}(v+\eta)=|v|^{\frac{4}{3}}v+[|v+\eta|^{\frac{4}{3}}(v+\eta)-|v|^{\frac{4}{3}}v],\\
	v(a)=w(a)-\eta(a).
	\end{cases}
\end{equation}
Moreover, thanks to \eqref{eq: isint}, we have $\|v(a)\|_{L_x^2}=M$, and we can view the system as a perturbation of 
\begin{equation}\label{eq: poeq}
		iv_{t}-\Delta v+i\langle t \rangle^{-2\epsilon_0}V^{2}v=|v|^{\frac{4}{3}}v
\end{equation}
and apply Lemma \ref{lem: stadnls}. Note the perturbation term $[|v+\eta|^{\frac{4}{3}}(v+\eta)-|v|^{\frac{4}{3}}v]$ depends on the solution $v$ itself.

To make it rigorous, we need to apply a bootstrap argument (continuity argument). Let $C_{M}$ be the largest implicit constant in Lemma \ref{lem: stadnls}. First apply Lemma \ref{lem: gwpdnls}, we find $B_{1}(M)$ so that if $\tilde{v}$ solves \eqref{eq: poeq} with initial data $v(a)$ then
\begin{equation}
	\|\tilde{v}\|_{L_{t}^{2}L_{x}^{6}\cap L_{t}^{\infty}L_{x}^{2}([a,b]\times \mathbb{R}^3)}\leq B_{1,M}.
\end{equation}
Next, choose $\epsilon_{1,M}$ with parameter $10B_{1,M}$ by Lemma \ref{lem: stadnls}. Then choose $\epsilon_{M}$ so that
\begin{equation}
	10^{10}B_{1,M}^{\frac{4}{3}}\epsilon_{M}\ll \epsilon_{1,M}, \text{ and } C_{M}	10^{10}B_{1, M}^{\frac{4}{3}}\epsilon_{M}\ll B_{1,M}.
\end{equation}
Then one can bootstrap the assumption $\|v\|_{L_{t}^{2}L_{x}^{6}\cap L_{t}^{\infty}L_{x}^{2}([a,t]\times \mathbb{R}^3)}\leq 5B_{1,M}$ into 
\begin{equation}
     \|v\|_{L_{t}^{2}L_{x}^{6}\cap L_{t}^{\infty}L_{x}^{2}([a,t]\times \mathbb{R}^3)}\leq \frac{5}{2} B_{1,M}.
\end{equation}
 Thus, choose $B_{M}:=2B_{1,M}$, and the desired estimate follows. 
\end{proof}

\subsection{Concluding the proof of Lemma \ref{lem: boot}}
Now we are ready to prove Lemma \ref{lem: boot}. 
Recall, within any interval $[a,b]$, we have the Duhamel Formula based on damped NLS as follows,
\begin{equation}\label{eq: duhamel}
u(t)=H(t,a)u(a)+i\int_{a}^{t}H(t,s)(u(s)|u(s)|^{\frac{4}{3}})ds+i\int_{a}^{t}H(t,s)(\langle s \rangle^{-\epsilon_0} Vu(s))dB_{s}.
\end{equation}
We also recall the choice of Strichartz pair $(\alpha_0,\beta_0)$ satisfies the condition \eqref{eq: close} and
\begin{equation}
M_{0}^{*}(s_0,t)=\sup_{s_0 \leq r_{1}\leq r_{2}\leq t}\big\|\int_{r_{1}}^{r_{2}}H(t,s)(\langle t \rangle^{-\epsilon_0} Vu)dB_{s}\big\|_{L_{x}^{\beta_0}}.
\end{equation}
Then we split the proof of Lemma \ref{lem: boot} into the following two lemmas.
\begin{lemma}\label{lem: boot1}
Consider $u$ solves \eqref{eq: snls}. We have,
\begin{equation}\label{eq: decayb1}
\|M_0^{*}(s_0,t)\|_{L_{\omega}^{\alpha_0}L_{t}^{\alpha_0}([s_0,T])}\lesssim_{\|u_{0}\|_{L^{\infty}_{\omega}L_{x}^{2}}, \alpha_{0}, \epsilon_{0}} \langle s_0 \rangle^{-\epsilon_{0}/2}\|u\|_{L^{\alpha_{0}}_{\omega}L_{t}^{\alpha_{0}}L_{x}^{\beta_{0}}([s_0,T])}.
\end{equation}
\end{lemma}
\begin{lemma}\label{lem: boot2}
Consider $u$ solves \eqref{eq: snls}. We have,
\begin{equation}\label{eq: gb1}
\|u\|_{L_{\omega}^{\alpha_0}L_{t}^{\alpha_0}L_{x}^{\beta_0}([0,T]\times \mathbb{R}^3)}\lesssim_{\|u_{0}\|_{L^{\infty}_{\omega}L_{x}^{2}}, \alpha_{0}, \epsilon_{0}} 1+\|M_{0}^{*}(t)\|_{L_{\omega}^{\alpha_0}L_{t}^{\alpha_0}([0,T])}.
\end{equation}

\end{lemma}
\begin{remark}
For convenience, we do not keep track of the implicit constant in the proof. We allow the constant depends on $\|u_{0}\|_{L^{\infty}_{\omega}L_{x}^{2}}, \alpha_{0}$ and $\epsilon_{0}$.
\end{remark}
\begin{proof}[Proof of Lemma \ref{lem: boot1}]
Via Burkholder inequality 
\begin{equation}
    \big\| M_0^{*}(s_0,t) \big\|^{\alpha_0}_{L_{\omega}^{\alpha_0}}\lesssim \mathbb{E}\big(\int_{s_0}^t \| H(t,s)(\langle s \rangle^{-\epsilon_0}Vu)\|^2_{L_x^{\beta_0}} ds\big)^{\frac{\alpha_0}{2}}. 
\end{equation}
Then using dispersive estimate and the fact $V$ is localized,
\begin{equation}
    \| H(t,s)V\langle s \rangle^{-\epsilon_0}u\|_{L_x^{\beta_0}}\lesssim (t-s)^{-3(\frac{1}{2}-\frac{1}{\beta_0})}\langle s \rangle^{-\epsilon_0}\|Vu\|_{L_x^{\beta_0^{'}}}\lesssim (t-s)^{-3(\frac{1}{2}-\frac{1}{\beta_0})}\langle s \rangle^{-\epsilon_0}\|u(s)\|_{L_{x}^{\beta_0}}.
\end{equation}
 To summarize, we have 
\begin{equation}
\|M_0^{*}(s_0,t)\|_{L_{\omega}^{\alpha_0}}\lesssim \left( \mathbb{E} (\int_{s_0}^{t}(t-s)^{-6(\frac{1}{2}-\frac{1}{\beta_0})} \langle s \rangle^{-2\epsilon_0} \|u(s)\|_{L_{x}^{\beta_0}}^{2}ds)^{\frac{\alpha_0}{2}}\right)^{\frac{1}{\alpha_0}}.
\end{equation}
Then we apply Minkowski inequality to obtain
\begin{equation}
\begin{aligned}
 &\big\| \big(\mathbb{E} (\int_{s_0}^{t}(t-s)^{-6(\frac{1}{2}-\frac{1}{\beta_0})}\langle s \rangle^{-2\epsilon_0}\|u(s)\|_{L_{x}^{\beta_0}}^{2}ds)^{\frac{\alpha_0}{2}}\big)^{\frac{1}{\alpha_0}}\big\|_{L_t^{\alpha_0}}=\big\| \big\|\int_{s_0}^{t}(t-s)^{-6(\frac{1}{2}-\frac{1}{\beta_0})}\langle s \rangle^{-2\epsilon_0}\|u(s)\|_{L_{x}^{\beta_0}}^{2}ds\big\|_{L_{\omega}^{\frac{\alpha_0}{2}}}^{\frac{1}{2}} \big\|_{L_t^{\alpha_0}}\\
 &\lesssim \big\| \big( \int_{s_0}^{t}(t-s)^{-6(\frac{1}{2}-\frac{1}{\beta_0})}\langle s \rangle^{-2\epsilon_0}\|u(s)\|_{L_{\omega}^{\alpha_0}L^{\beta_0}_x}^{2}ds \big)^{\frac{1}{2}} \big\|_{L_t^{\alpha_0}}\\
 &\lesssim \langle s_0 \rangle^{-\frac{\epsilon_0}{2}} \big\|  \int_{s_0}^{t}(t-s)^{-6(\frac{1}{2}-\frac{1}{\beta_0})}\langle s \rangle^{-\epsilon_0}\|u(s)\|_{L_{\omega}^{\alpha_0}L^{\beta_0}_x}^{2}ds  \big\|^{\frac{1}{2}}_{L_t^{\frac{\alpha_0}{2}}}.
 \end{aligned}
\end{equation}
Here we recall
\begin{equation}
    -6(\frac{1}{2}-\frac{1}{\beta_0}) =-1+2\eta_0,
\end{equation}
according to \eqref{eq: close}. Now we split the time interval $[s_0,t]$ into two parts and write
\begin{equation}
\aligned
  &\int_{s_0}^{t}(t-s)^{-6(\frac{1}{2}-\frac{1}{\beta_0})}\langle s \rangle^{-\epsilon_0}\|u(s)\|_{L_{\omega}^{\alpha_0}L^{\beta_0}_x}^{2}ds=\int_{s_0}^{t}(t-s)^{-1+2\eta_0}\langle s \rangle^{-\epsilon_0}\|u(s)\|_{L_{\omega}^{\alpha_0}L^{\beta_0}_x}^{2}ds\\
  &=   \int_{s_0}^{t-1}(t-s)^{-1+2\eta_0}\langle s \rangle^{-\epsilon_0}\|u(s)\|_{L_{\omega}^{\alpha_0}L^{\beta_0}_x}^{2}ds 
  + \int_{t-1}^{t}(t-s)^{-1+2\eta_0}\langle s \rangle^{-\epsilon_0}\|u(s)\|_{L_{\omega}^{\alpha_0}L^{\beta_0}_x}^{2}ds. 
  \endaligned
\end{equation}
Next we will apply Young's inequality for the two parts respectively.

For the first part, we denote $h_1(s)=1_{s\geq 1}(s)^{-1+2\eta_0}$ and $f(s)=1_{s\geq s_0} \langle s \rangle^{-\epsilon_0}\|u(s)\|_{L_{\omega}^{\alpha_0}L^{\beta_0}_x}^{2}$. Then the integral can be regarded as $h_1 \ast f(t)$. And it suffices to consider
\begin{equation}
    \|(h_1 \ast f)(t)\|_{L_t^{\frac{\alpha_0}{2}}}^{\frac{1}{2}}.
\end{equation}
By Young's inequality and the H\"older, we have
\begin{equation}
\aligned
    \|(h_1 \ast f)(t)\|_{L_t^{\frac{\alpha_0}{2}}}^{\frac{1}{2}} \lesssim \|h_1\|^{\frac{1}{2}}_{L_t^{p_1}}\|f\|^{\frac{1}{2}}_{L_t^{q_1}}\lesssim \|u(t)\|_{L_{\omega}^{\alpha_0}L_t^{\alpha_0}L_x^{\beta_0}}.
    \endaligned
\end{equation}
We note that we can choose $p_1,q_1$ satisfies (according to \eqref{eq: close})
\begin{equation}
    \aligned
    &\frac{2}{\alpha_0}+1=\frac{1}{p_1}+\frac{1}{q_1}, \quad \textmd{(Young's inequality)}\\
    &p_1 \cdot 6(\frac{1}{2}-\frac{1}{\beta_0})>1,\quad \textmd{(Integrability)} \\
    &\frac{1}{q_1}=\frac{1}{\tilde{q}_{1}}+\frac{2}{\alpha_0},\quad \textmd{(the H\"older)}\\
    &\tilde{q}_{1} \cdot \epsilon_0 >1,\quad \textmd{(Integrability)}.
    \endaligned
\end{equation}
For the second part, it is similar and easier. We denote $h_2(s)={1}_{0<s\leq 1}(s)^{-1+2\eta_0}$ and $f(s)={1}_{s\geq s_0} \langle s \rangle^{-\epsilon_0}\|u(s)\|_{L_{\omega}^{\alpha_0}L^{\beta_0}_x}^{2}$. Then the integral can be regarded as $h_2 \ast f(t)$. And it suffices to consider
\begin{equation}
    \|(h_2 \ast f)(t)\|_{L_t^{\frac{\alpha_0}{2}}}^{\frac{1}{2}}.
\end{equation}
By Young's inequality and the H\"older, we have
\begin{equation}
    \|(h_2 \ast f)(t)\|_{L_t^{\frac{\alpha_0}{2}}}^{\frac{1}{2}}\lesssim \|h_2\|^{\frac{1}{2}}_{L_t^{1}}\|f\|^{\frac{1}{2}}_{L_t^{\frac{\alpha_0}{2}}}\lesssim \|u(t)\|_{L_{\omega}^{\alpha_0}L_t^{\alpha_0}L_x^{\beta_0}}.
\end{equation}
We conclude that the above estimates imply,
\begin{equation}
\|M_0^{*}(s_0,t)\|_{L_{\omega}^{\alpha_0}L_t^{\alpha_0}} \lesssim \langle s_0 \rangle^{-\frac{\epsilon_{0}}{2}} \|u(t)\|_{L_{\omega}^{\alpha_0}L_t^{\alpha_0}L_x^{\beta_0}}.
\end{equation}
(According to the Young and the H\"older) At last, we note that it suffices to choose $\beta_0$ satisfying
\begin{equation}
    2\epsilon_0+6(\frac{1}{2}-\frac{1}{\beta_0})>1,
\end{equation}
which means
\begin{equation}
    \frac{3}{1+\epsilon_0}<\beta_0<3.
\end{equation}
Our assumption \eqref{eq: close} is enough.
\end{proof}

\begin{proof}[Proof of Lemma \ref{lem: boot2}]
Recall $M=\|u\|_{L_{\omega}^{\infty}L_{t}^{\infty}L_{x}^{2}}$ are fixed all time.
 Let $\epsilon_{M}$ be as in Proposition \ref{prop: modistab}. We will denote $M^{*}(0,t)$ by $M_0^{*}(t)$.
 
First note 
  \begin{equation}\label{eq: lebe}
	\sum_{A}A^{\alpha_0} \mathbb{P}\big(\omega:\|M_0^{*}(t)\|_{L_{t}^{\alpha_0}}\sim A \big) \lesssim \|M_0^{*}(t)\|_{L_{\omega}^{\alpha_0}L_{t}^{\alpha_0}}^{\alpha_0}+1.
\end{equation}
where $A$ ranges over all dyadic integers.

We fix $\omega$ such that $\|M_0^{*}(t)\|_{L_{t}^{\alpha_0}} \sim A$, and we proceed in a deterministic way.

 We  divide\footnote{We abuse notation a bit and don't distinguish $[a,\infty]$ and $[a,\infty)$.}  $[0,\infty]$ into $\cup_{j=1}^{J}[t_{j}, t_{j+1}]$ so that
\begin{equation}\label{eq: intervalnumber}
		J\lesssim  \frac{A^{\alpha_0}}{\epsilon_{M}^{\alpha_0}}+1  \text{ and } \|M_0^{*}(t)\|_{L_{t}^{\alpha_0}([t_{j}, t_{j+1}])}\leq \epsilon_{M}.
\end{equation}	
We write  Duhamel Formula for $u$ in $[t_{j}, t_{j+1}]$ based on damped NLS as follows, 
\begin{equation}
	u(t)=H(t,t_{j})u(t_{j})+i\int_{t_{j}}^{t}H(t,s)N(u(s))+\eta(t),
\end{equation}
where 
\begin{equation}
\eta(t):=\eta(t_{j},t):=i\int_{t_{j}}^{t}H(t,s)(\langle s \rangle^{-\epsilon_{0}}Vu)dB_{s}.
\end{equation}
Note that 
\begin{equation}
	\|\eta(t)\|_{L_{x}^{\beta_0}}\leq M_0^{*}(t),
\end{equation}
and in particular
\begin{equation}
\|\eta(t)\|_{L_{t}^{\alpha_{0}}([t_{j},t_{j+1}])}\leq \epsilon_{M}.
\end{equation}

Thus, by Proposition \ref{prop: modistab}, we have

\begin{equation}
\|u\|_{L_{t}^{\alpha_{0}}L_{x}^{\beta_{0}}([t_{j},t_{j+1}])}\leq B_{M}.
\end{equation}

Thus, we have for all $\omega$ such that $\|M_{0}^{*}(t)\|\sim A$, 
\begin{equation}
\|u\|_{L_{t}^{\alpha_{0}}L_{x}^{\beta_{0}}}^{\alpha_{0}}\lesssim B_{M}^{\alpha_0}[\frac{A^{\alpha_0}}{\epsilon_{M}^{\alpha_0}}+1 ]. 
\end{equation}

Integrate in $\omega$ with \eqref{eq: lebe} and desired result follows.
\end{proof}

\section{Proof of Lemma \ref{lem: m1}, Lemma \ref{lem: splitintotwo}}\label{sec: fisrtsplit}
We start with the proof of Lemma \ref{lem: m1}.

Recall 
\begin{equation}
M_1^{*}(t)=\sup_{0\leq r_{1}\leq r_{2}\leq t}\big\|\int_{r_{1}}^{r_{2}}H(t,s)(\langle s \rangle^{-\epsilon_0} Vu)dB_{s}\big\|_{L_{x}^{\beta_1}}.
\end{equation}
\begin{proof}
 We choose $\beta_{1}$ close to 3 enough so that 
 \begin{equation}\label{eq: b1close3}
 -6(\frac{1}{2}-\frac{1}{\beta_1}) =-1+2\eta_1. 
\end{equation}
Via Burkholder inequality 
\begin{equation}
    \big\| M_1^{*}(t) \big\|^{\alpha_1}_{L_{\omega}^{\alpha_1}}\lesssim_{\alpha_1} \mathbb{E}\big(\int_a^t \| H(t,s)(\langle s \rangle^{-\epsilon_0}Vu)\|^2_{L_x^{\beta_1}} ds\big)^{\frac{\alpha_1}{2}}. 
\end{equation}
Then using dispersive estimate and the fact $V$ is localized,
\begin{equation}
    \| H(t,s)(\langle s \rangle^{-\epsilon_0}Vu)\|_{L_x^{\beta_1}}\lesssim (t-s)^{-3(\frac{1}{2}-\frac{1}{\beta_1})}\langle s \rangle^{-\epsilon_0}\|Vu\|_{L_x^{\beta_1^{'}}}\lesssim (t-s)^{-3(\frac{1}{2}-\frac{1}{\beta_1})}\langle s \rangle^{-\epsilon_0}\|u(s)\|_{L_{x}^{\beta_0}}.
\end{equation}
To summarize, we have 
\begin{equation}
\|M_1^{*}(t)\|_{L_{\omega}^{\alpha_1}}\lesssim \left( \mathbb{E} (\int_{0}^{t}(t-s)^{-6(\frac{1}{2}-\frac{1}{\beta_1})} \langle s \rangle^{-2\epsilon_0} \|u(s)\|_{L_{x}^{\beta_0}}^{2}ds)^{\frac{\alpha_1}{2}}\right)^{\frac{1}{\alpha_1}}.
\end{equation}
Then we apply Minkowski inequality to obtain
\begin{equation}
\begin{aligned}
 &\big\| \big(\mathbb{E} (\int_{0}^{t}(t-s)^{-6(\frac{1}{2}-\frac{1}{\beta_1})}\langle s \rangle^{-2\epsilon_0}\|u(s)\|_{L_{x}^{\beta_0}}^{2}ds)^{\frac{\alpha_1}{2}}\big)^{\frac{1}{\alpha_1}}\big\|_{L_t^{\alpha_1}}=\big\| \big\|\int_{0}^{t}(t-s)^{-6(\frac{1}{2}-\frac{1}{\beta_1})}\langle s \rangle^{-2\epsilon_0}\|u(s)\|_{L_{x}^{\beta_0}}^{2}ds\big\|_{L_{\omega}^{\frac{\alpha_1}{2}}}^{\frac{1}{2}} \big\|_{L_t^{\alpha_1}}\\
 &\lesssim \big\| \big( \int_{0}^{t}(t-s)^{-6(\frac{1}{2}-\frac{1}{\beta_1})}\langle s \rangle^{-2\epsilon_0}\|u(s)\|_{L_{\omega}^{\alpha_1}L^{\beta_0}_x}^{2}ds \big)^{\frac{1}{2}} \big\|_{L_t^{\alpha_1}} \lesssim \|u(t)\|_{L_{t}^{\alpha_0}L_{\omega}^{\alpha_1}L_x^{\beta_0}} \lesssim \|u(t)\|_{L_{\omega}^{\alpha_0}L_{t}^{\alpha_0}L_x^{\beta_0}}.
 \end{aligned}
\end{equation}
Now we split the time interval $[0,t]$ into two parts and write
\begin{equation}
\aligned
  &\int_{0}^{t}(t-s)^{-6(\frac{1}{2}-\frac{1}{\beta_1})}\langle s \rangle^{-\epsilon_0}\|u(s)\|_{L_{\omega}^{\alpha_1}L^{\beta_0}_x}^{2}ds=\int_{0}^{t}(t-s)^{-1+2\eta_1}\langle s \rangle^{-\epsilon_0}\|u(s)\|_{L_{\omega}^{\alpha_1}L^{\beta_0}_x}^{2}ds\\
  &=   \int_{0}^{t-1}(t-s)^{-1+2\eta_1}\langle s \rangle^{-\epsilon_0}\|u(s)\|_{L_{\omega}^{\alpha_1}L^{\beta_0}_x}^{2}ds 
  + \int_{t-1}^{t}(t-s)^{-1+2\eta_1}\langle s \rangle^{-\epsilon_0}\|u(s)\|_{L_{\omega}^{\alpha_1}L^{\beta_0}_x}^{2}ds. 
  \endaligned
\end{equation}
Next we will apply Young's inequality for the two parts respectively.

For the first part, we denote $h_1(s)={1}_{s\geq 1}(s)^{-1+2\eta_1}$ and $f(s)={1}_{s\geq 0} \langle s \rangle^{-\epsilon_0}\|u(s)\|_{L_{\omega}^{\alpha_1}L^{\beta_0}_x}^{2}$. Then the integral can be regarded as $h_1 \ast f(t)$. And it suffices to consider
\begin{equation}
    \|(h_1 \ast f)(t)\|_{L_t^{\frac{\alpha_1}{2}}}^{\frac{1}{2}}.
\end{equation}
By Young's inequality, the H\"older and Minkowski inequality, we have
\begin{equation}
\aligned
    \|(h_1 \ast f)(t)\|_{L_t^{\frac{\alpha_1}{2}}}^{\frac{1}{2}} \lesssim \|h_1\|^{\frac{1}{2}}_{L_t^{p_1}}\|f\|^{\frac{1}{2}}_{L_t^{q_1}}\lesssim \|u(t)\|_{L_{t}^{\alpha_0}L_{\omega}^{\alpha_1}L_x^{\beta_0}} \lesssim \|u(t)\|_{L_{\omega}^{\alpha_1}L_{t}^{\alpha_0}L_x^{\beta_0}} \lesssim \|u(t)\|_{L_{\omega}^{\alpha_0}L_{t}^{\alpha_0}L_x^{\beta_0}}.
    \endaligned
\end{equation}
We emphasize that for the last inequality, the H\"older is applied since $\alpha_1<\alpha_0$ and the probability space has finite measure. We note that we can choose $p_1,q_1$ satisfies
\begin{equation}
    \aligned
    &\frac{2}{\alpha_1}+1=\frac{1}{p_1}+\frac{1}{q_1}, \quad \textmd{(Young's inequality)}\\
    &p_1 \cdot 6(\frac{1}{2}-\frac{1}{\beta_1})>1,\quad \textmd{(Integrability)} \\
    &\frac{1}{q_1}=\frac{1}{\tilde{q}_{1}}+\frac{2}{\alpha_0},\quad \textmd{(the H\"older)}\\
    &\tilde{q}_{1} \cdot \epsilon_0 >1,\quad \textmd{(Integrability)}.
    \endaligned
\end{equation}
For the second part, it is similar and easier. We denote $h_2(s)={1}_{0<s\leq 1}(s)^{-1+2\eta_1}$ and $f(s)={1}_{s\geq 0} \langle s \rangle^{-\epsilon_0}\|u(s)\|_{L_{\omega}^{\alpha_1}L^{\beta_0}_x}^{2}$. Then the integral can be regarded as $h_2 \ast f(t)$. And it suffices to consider
\begin{equation}
    \|(h_2 \ast f)(t)\|_{L_t^{\frac{\alpha_1}{2}}}^{\frac{1}{2}}.
\end{equation}
By Young's inequality, the H\"older and Minkowski inequality, we have
\begin{equation}
    \|(h_2 \ast f)(t)\|_{L_t^{\frac{\alpha_1}{2}}}^{\frac{1}{2}}\lesssim \|h_2\|^{\frac{1}{2}}_{L_t^{1}}\|f\|^{\frac{1}{2}}_{L_t^{\frac{\alpha_1}{2}}}\lesssim \|u(t)\|_{L_{t}^{\alpha_0}L_{\omega}^{\alpha_1}L_x^{\beta_0}} \lesssim \|u(t)\|_{L_{\omega}^{\alpha_1}L_{t}^{\alpha_0}L_x^{\beta_0}} \lesssim \|u(t)\|_{L_{\omega}^{\alpha_0}L_{t}^{\alpha_0}L_x^{\beta_0}}.
\end{equation}
In summary, the above estimates imply
\begin{equation}
\|M_1^{*}(t)\|_{L_{\omega}^{\alpha_1}L_{t}^{\alpha_1}}\lesssim 1.
\end{equation}
\end{proof}
We now turn to the proof of Lemma \ref{lem: splitintotwo}.
\begin{proof}
As in the proof of Lemma \ref{lem: boot2}, Let $\epsilon_{M}$ be as in Proposition \ref{prop: modistab}.  

Fix $\omega$ such that $\|M^{*}_{0}(t)\|_{L_{t}^{\alpha_{0}}}\sim A$, and we 
  divide $[0,\infty)$ into $\cup_{j=1}^{J}[t_{j}, t_{j+1}]$ so that \eqref{eq: intervalnumber} holds.

In each $[t_{j}, t_{j+1}]$
\begin{equation}
	u(t)=H(t,t_{j})u(t_{j})+i\int_{t_{j}}^{t}H(t,s)N(u(s))+\eta(t_{j},t)
\end{equation}
where 
\begin{equation}
\eta(t_{j},t):=i\int_{t_{j}}^{t}H(t,s)(\langle s \rangle^{-\epsilon_{0}}Vu)dB_{s}.
\end{equation}

We will let $u_{2}:=\eta(t_{j},t)$ on each $[t_{j}, t_{j+1}]$ and $u_{1}=u-u_{2}$.

Note one has pointwise control 
\begin{equation}
\|u_{2}\|_{L_{x}^{\beta_{1}}}\leq M^{*}_{1}(t),
\end{equation}
and thus \eqref{eq: fisrtestimateforu2} holds.

Thus, by Proposition \ref{prop: modistab} on $[t_{j},t_{j+1}]$, we have
\begin{equation}\label{eq: bdjj1}
	\|u_1\|_{L_{t}^{2}L_{x}^{6}\cap {L_{t}^{\infty}L_{x}^{2}}([t_{j}, t_{j+1}]\times \mathbb{R}^3)}\leq B_{M}.
\end{equation}
Summing up all the  time subintervals, recall \eqref{eq: intervalnumber}, we have
\begin{equation}
\|u_1\|^{2}_{L_{t}^{2}L_{x}^{6}([0,\infty])}\lesssim B_{M}^{2}[\frac{A^{\alpha_0}}{\epsilon_{M}^{\alpha_0}}+1 ].
\end{equation}
Integration in $\omega$, and recall \eqref{eq: lebe}
\begin{equation}\label{eq: sum1}
	\|u_1\|_{L_{\omega}^{2}L_{t}^{2}L_{x}^{6}}^{2}\lesssim \sum_{A}B_{M}^{2}[\frac{A^{\alpha_0}}{\epsilon_{M}^{\alpha_0}}+1 ] 	\mathbb{P}\big(\omega:\|M_0^{*}(t)\|_{L_{t}^{\alpha_0}}\sim A \big) \lesssim \|M^{*}_{0}(t)\|_{L_{\omega}^{\alpha_{0}}L_{t}^{\alpha_{0}}}^{\alpha_{0}}+1.
\end{equation}
Thus, \eqref{eq: firstestiamteforu1} follows.

We also point out via \eqref{eq: bdjj1}, we have $u_{1}\in L_{t}^{\infty}L_{t}^{2}L_{x}^{2}$, and  \eqref{eq: firstimprove} follows from interpolation and the fact $V$ is localized.
\end{proof}

\section{Proof of Lemma \ref{lem: mn}, Lemma \ref{lem: splitagain} and the derivation of scattering asymptotics}\label{sec: conclude}
The proof of Lemma \ref{lem: mn} is similar as Lemma \ref{lem: m1}, we simply replace the control $\|u\|_{L_{\omega}^{\alpha_{0}}L_{t}^{\alpha_{0}}}$ via
$\|V^{\frac{1}{2}}u\|_{L_{\omega}^{\alpha_{n}}L_{t}^{\alpha_{n}}L_{x}^{2}}$ and observe $Vu=V^{\frac{1}{2}}V^{\frac{1}{2}}u$.

The proof of Lemma \ref{lem: splitagain} is similar to the proof of Lemma \ref{lem: splitintotwo}, we still based on $M^{*}_{0}$ to divide small intervals, but we use $M_{n}$ to control $u_{2}$.

We finally explain how Lemma \ref{lem: summary} implies the desired scattering dynamic.
We start with the proof of \eqref{eq: scattering3d}.
Write (the usual) Duhamel Formula as 
\begin{equation}
	\begin{aligned}
	u(t,x)=&S(t)u_{0}-i\int_{0}^{t}S(t-s)N(u(s))ds\\
	-&i\int_{0}^{t}S(t-s)(V(x)u(s)\langle s \rangle^{-\epsilon_0})dB_{s}-\frac{1}{2}\int_{0}^{t}S(t-s)(\langle s \rangle^{-2\epsilon_0}V^{2}u(s))ds.
	\end{aligned}
\end{equation}

Note $e^{it\Delta}$ is unitary and 
\begin{equation}
\begin{aligned}
	e^{-it\Delta}u(t,x)=&S(t)u_{0}-i\int_{0}^{t}S(-s)N(u(s))ds\\
	-&i\int_{0}^{t}S(-s)(V(x)u(s)\langle s \rangle^{-\epsilon_0})dB_{s}-\frac{1}{2}\int_{0}^{t}S(-s)(\langle s \rangle^{-2\epsilon_0}V^{2}u(s))ds.
	\end{aligned}
\end{equation}
Since $N(u)$ is in $L_{\omega}^{\frac{3\alpha_{0}}{7}}L_{t}^{\frac{3\alpha_{0}}{7}}L_{x}^{\frac{3\beta_{0}}{7}}$ and $(\frac{3\alpha_{0}}{7}, \frac{3\beta_{0}}{7})$ is in the dual of Strichartz pair, thus $i\int_{0}^{t}S(-s)N(u(s))ds$ converges in $L_{\omega}^{\rho^{*}}L_{x}^{2}$ for some $\rho^{*}>1.$

Similarly, we have $\langle s \rangle^{-2\epsilon_{0}}V^{2}u$ is in $L_{\omega}^{2}L_{t}^{2}L_{x}^{\frac{6}{5}}$  and $\langle s \rangle^{-\epsilon_{0}}Vu$ is in $L_{\omega}^{2}L_{t}^{2}L_{x}^{2}$ since $V$ is localized and thanks to \eqref{eq: fffinal}.

The first fact handles $-\frac{1}{2}\int_{0}^{t}S(-s)(\langle s \rangle^{-2\epsilon_0}V^{2}u(s))ds.$ via Strichartz estimate and the second fact handles 
$i\int_{0}^{t}S(-s)(V(x)u(s)\langle s \rangle^{-\epsilon_0})dB_{s}$ via Burkholder.

Note since $u\in L_{\omega}^{\infty}L_{t}^{\infty}L_{x}^{2}$, we have $u^{+}\in L_{\omega}^{\infty}L_{x}^{2}$.

Thus \eqref{eq: scattering3d} follows and so is \eqref{eq: conpro}.

To upgrade to the almost sure convergence, observe sine almost surely $N(u)$ is 
in $L_{t}^{\frac{3\alpha_{0}}{7}}L_{x}^{\frac{3\beta_{0}}{7}}$ and $\langle s \rangle^{-2\epsilon_{0}}V^{2}u$ is in $L_{t}^{2}L_{x}^{\frac{6}{5}}$, we only need to handle the Stochastic term.
But it is a martingale term and Burkholder gives
\begin{equation}
\lim_{t\rightarrow\infty}\big\|\sup_{a\geq b\geq t}\int_{a}^{b}S(-s)(\langle s \rangle^{-\epsilon_{0}}u)dB_{s}\big\|_{L_{\omega}^{2}L_{x}^{2}}=0.
\end{equation}
Thus, almost sure convergence follows. q.e.d
\bibliographystyle{amsplain}
\bibliographystyle{plain}
\bibliography{BG}

\providecommand{\bysame}{\leavevmode\hbox to3em{\hrulefill}\thinspace}
\providecommand{\MR}{\relax\ifhmode\unskip\space\fi MR }
\providecommand{\MRhref}[2]{%
  \href{http://www.ams.org/mathscinet-getitem?mr=#1}{#2}
}
\providecommand{\href}[2]{#2}
\begin{thebibliography}{10}

\bibitem{BRZ14}
Viorel Barbu, Michael R\"{o}ckner, and Deng Zhang, \emph{Stochastic nonlinear
  {S}chr\"{o}dinger equations with linear multiplicative noise: rescaling
  approach}, J. Nonlinear Sci. \textbf{24} (2014), no.~3, 383--409.
  \MR{3215081}

\bibitem{BRZ16}
\bysame, \emph{Stochastic nonlinear {S}chr\"{o}dinger equations}, Nonlinear
  Anal. \textbf{136} (2016), 168--194. \MR{3474409}

\bibitem{B}
Zdzis\l~aw Brze\'{z}niak, \emph{On stochastic convolution in {B}anach spaces
  and applications}, Stochastics Stochastics Rep. \textbf{61} (1997), no.~3-4,
  245--295. \MR{1488138}

\bibitem{BP}
Zdzis\l~aw Brze\'{z}niak and Szymon Peszat, \emph{Space-time continuous
  solutions to {SPDE}'s driven by a homogeneous {W}iener process}, Studia Math.
  \textbf{137} (1999), no.~3, 261--299. \MR{1736012}

\bibitem{Bur73}
D.~L. Burkholder, \emph{Distribution function inequalities for martingales},
  Ann. Probability \textbf{1} (1973), 19--42. \MR{365692}

\bibitem{BDG73}
D.~L. Burkholder, B.~J. Davis, and R.~F. Gundy, \emph{Integral inequalities for
  convex functions of operators on martingales}, Proceedings of the {S}ixth
  {B}erkeley {S}ymposium on {M}athematical {S}tatistics and {P}robability
  ({U}niv. {C}alifornia, {B}erkeley, {C}alif., 1970/1971), {V}ol. {II}:
  {P}robability theory, 1972, pp.~223--240. \MR{0400380}

\bibitem{cazenave2003semilinear}
Thierry Cazenave, \emph{Semilinear {S}chr{\"o}dinger equations}, vol.~10,
  American Mathematical Soc., 2003.

\bibitem{cazenave1990cauchy}
Thierry Cazenave and Fred~B Weissler, \emph{The cauchy problem for the critical
  nonlinear {S}chr{\"o}dinger equation in hs}, Nonlinear Analysis: Theory,
  Methods \& Applications \textbf{14} (1990), no.~10, 807--836.

\bibitem{colliander2008global}
James Colliander, Markus Keel, Gigiola Staffilani, Hideo Takaoka, and Terence
  Tao, \emph{Global well-posedness and scattering for the energy-critical
  nonlinear {S}chr{\"o}dinger equation in {$R^3$}}, Annals of Mathematics
  (2008), 767--865.

\bibitem{dBD99}
A.~de~Bouard and A.~Debussche, \emph{A stochastic nonlinear {S}chr\"{o}dinger
  equation with multiplicative noise}, Comm. Math. Phys. \textbf{205} (1999).

\bibitem{dBD03}
\bysame, \emph{The stochastic nonlinear {S}chr\"{o}dinger equation in {$H^1$}},
  Stochastic Anal. Appl. \textbf{21} (2003), no.~1, 97--126. \MR{1954077}

\bibitem{dodson2012global}
Benjamin Dodson, \emph{Global well-posedness and scattering for the defocusing,
  -critical nonlinear {S}chr{\"o}dinger equation when {$d\geq 3$}}, Journal of
  the American Mathematical Society \textbf{25} (2012), no.~2, 429--463.

\bibitem{dodson2015global}
\bysame, \emph{Global well-posedness and scattering for the mass critical
  nonlinear {S}chr{\"o}dinger equation with mass below the mass of the ground
  state}, Advances in Mathematics \textbf{285} (2015), 1589--1618.

\bibitem{Dodson1}
\bysame, \emph{Global well-posedness and scattering for the defocusing, {$L^2$}
  critical, nonlinear {S}chr\"{o}dinger equation when {$d=1$}}, Amer. J. Math.
  \textbf{138} (2016), no.~2, 531--569. \MR{3483476}

\bibitem{Dodson2}
\bysame, \emph{Global well-posedness and scattering for the defocusing,
  {$L^2$}-critical, nonlinear {S}chr\"{o}dinger equation when {$d=2$}}, Duke
  Math. J. \textbf{165} (2016), no.~18, 3435--3516. \MR{3577369}

\bibitem{FX}
Chenjie Fan and Weijun Xu, \emph{Decay of the stochastic linear schrödinger
  equation in $d\geq 3$ with small multiplicative noise}, arXiv preprint
  arXiv:1812.06661 (2018).

\bibitem{fan2018global}
Chenjie Fan and Weijun Xu, \emph{Global well-posedness for the mass-critical
  stochastic nonlinear schr{\"o}dinger equation on {$\mathbb{R}$}: general
  {${L^2}$} data}, arXiv preprint arXiv:1807.04402 (2018).

\bibitem{fan2019wong}
\bysame, \emph{A wong-zakai theorem for mass critical nls}, arXiv preprint
  arXiv:1906.06616 (2019).

\bibitem{HRZ}
Sebastian Herr, Michael R\"{o}ckner, and Deng Zhang, \emph{Scattering for
  stochastic nonlinear {S}chr\"{o}dinger equations}, Comm. Math. Phys.
  \textbf{368} (2019).

\bibitem{Hor}
Fabian Hornung, \emph{The nonlinear stochastic {S}chr\"{o}dinger equation via
  stochastic {S}trichartz estimates}, J. Evol. Equ. \textbf{18} (2018), no.~3,
  1085--1114. \MR{3859442}

\bibitem{JSS}
J.-L. Journ\'{e}, A.~Soffer, and C.~D. Sogge, \emph{Decay estimates for
  {S}chr\"{o}dinger operators}, Comm. Pure Appl. Math. \textbf{44} (1991),
  no.~5, 573--604. \MR{1105875}

\bibitem{Oh}
T.~Oh and M.~Okamoto, \emph{On the stochastic nonlinear schr\"odinger equations
  at critical regularities}, Stoch PDE: Anal Comp (2020) (2020).

\bibitem{Schlagsurvey}
W.~Schlag, \emph{Dispersive estimates for {S}chr\"{o}dinger operators: a
  survey}, Mathematical aspects of nonlinear dispersive equations, Ann. of
  Math. Stud., vol. 163, Princeton Univ. Press, Princeton, NJ, 2007,
  pp.~255--285. \MR{2333215}

\bibitem{tao2006nonlinear}
Terence Tao, \emph{Nonlinear dispersive equations: local and global analysis},
  vol. 106, American Mathematical Soc., 2006.

\bibitem{weinstein1983nonlinear}
Michael~I Weinstein, \emph{Nonlinear {S}chr{\"o}dinger equations and sharp
  interpolation estimates}, Communications in Mathematical Physics \textbf{87}
  (1983), no.~4, 567--576.

\bibitem{YDeng1}
A.~Nahmod Y.~Deng and H.~Yue, \emph{Invariant gibbs measures and global strong
  solutions for nonlinear schr\"ödinger equations in dimension two}, arXiv
  preprint arXiv:1910.08492 (2019).

\bibitem{YDeng2}
\bysame, \emph{Random tensors, propagation of randomness, and nonlinear
  dispersive equations}, arXiv preprint arXiv:2006.09285 (2020).

\bibitem{DZhang}
D.~Zhang, \emph{Stochastic nonlinear schr\"odinger equations in the defocusing
  mass and energy critical cases}, arXiv preprint arXiv:1811.00167 (2018).

\end{thebibliography}

\end{document}